\documentclass[a4paper, 11pt]{amsart}
\usepackage[utf8]{inputenc}
\usepackage[english]{babel}
\usepackage[T1]{fontenc}
\usepackage{amsmath, amssymb, amsfonts}
\usepackage{enumerate}
\usepackage{graphicx}
\usepackage{comment}
\usepackage[hmargin=1.5cm, vmargin=2.5cm]{geometry}
\usepackage{tikz}
\usepackage[all,cmtip]{xy}
\usepackage[colorlinks=true, citecolor=blue]{hyperref}

\title{Free wreath products with amalgamation}
\author{Amaury Freslon}
\email{amaury.freslon@math.u-psud.fr}
\address{Laboratoire de Mathématique d’Orsay, CNRS, Université Paris-Saclay, 91405 Orsay, France.}
\date{}

\theoremstyle{plain}

\newtheorem{thm}{Theorem}[section]
\newtheorem{lem}[thm]{Lemma}
\newtheorem{prop}[thm]{Proposition}
\newtheorem{cor}[thm]{Corollary}

\theoremstyle{definition}
\newtheorem{de}[thm]{Definition}

\newtheorem{rem}[thm]{Remark}

\DeclareMathOperator{\cd}{cd}

\DeclareMathOperator{\id}{id}
\DeclareMathOperator{\Irr}{Irr}
\DeclareMathOperator{\Mor}{Mor}

\DeclareMathOperator{\Span}{span}

\newcommand{\A}{\mathcal{A}}

\newcommand{\C}{\mathbb{C}}
\newcommand{\CC}{\mathcal{C}}
\newcommand{\D}{\Delta}

\newcommand{\G}{\mathbb{G}}
\newcommand{\HH}{\mathbb{H}}
\newcommand{\K}{\mathbb{K}}

\newcommand{\N}{\mathbf{N}}

\newcommand{\Z}{\mathbf{Z}}

\renewcommand{\O}{\mathcal{O}}
\renewcommand{\SS}{\mathcal{S}}

\begin{document}

\begin{abstract}
We define and study a notion of free wreath product with amalgamation for compact quantum groups. These objects were already introduced in the case of duals of discrete groups under the name ``free wreath products of pairs'' in a previous work of ours. We give several equivalent descriptions and use them to establish properties like residual finiteness, the Haagerup property or a smash product decomposition.
\end{abstract}

\maketitle

\section{Introduction}

Compact quantum groups were introduced by S.L. Woronowicz in \cite{woronowicz1987compact} and \cite{woronowicz1995compact} as a generalisation of compact groups, and in particular of compact Lie groups, in the spirit of non-commutative geometry. The study of these objects has continued ever since and revealed unsuspected connections to other fields like combinatorics \cite{banica2009liberation}, free probability \cite{kostler2009noncommutative} or graph theory \cite{mancinska2019quantum}. One essential aspect of these advances is that they were all related to the search for new examples of compact quantum groups.

Another direction of research on this topic is a vast classification program for a subclass of compact quantum groups that can be described combinatorially, the \emph{easy} \cite{banica2009liberation} (or more generally \emph{partition} \cite{freslon2014partition}) quantum groups. The study of various types of invariants and constructions for the purpose of their classification led to the discovery of new families of examples of compact quantum groups and to new insight into previously known ones. The present work is an instance of this phenomenon. We will be interested in compact quantum groups which were discovered in the course of the classification of non-crossing partition quantum groups on two self-adjoint colours in \cite{freslon2017noncrossing}, and which turned out to be generalizations of objects introduced by J. Bichon in \cite{bichon2004free} in relation to quantum automorphism groups of graphs.

The \emph{free wreath product with amalgamation} is a generalization of the free wreath product where instead of considering a single compact quantum group $\G$ and an integer $N$, we consider a dual inclusion of quantum groups (see Section \ref{sec:general} for details). The definition came out of a combinatorial description of the representation theory of these objects when $\G$ is the dual of a discrete group, but free wreath products are know to enjoy several equivalent descriptions, each of them suited to the study of some specific properties. In the present work, we will give similar descriptions and deduce from them several structure results on free wreath products of pairs without always resorting to the partition picture.

Let us now describe more precisely the contents of the paper. After some preliminaries, we will give in Section \ref{sec:general} an abstract definition of free wreath products with amalgamation (Definition \ref{de:freewreathproductwithamalgamation}) together with elementary examples. We will then focus on the case of discrete groups to give another characterization in terms of universal $*$-algebras (Proposition \ref{prop:universaldiscrete}) and use it to deduce several results concerning monoidal equivalence and approximation properties. We will conclude that first part with a specific result concerning discrete abelian groups, which in turn sheds light on the classical counterpart of free wreath products with amalgamation, which we will simply call wreath product with amalgamation.

In Section \ref{sec:partitions}, we make the connection with the original definition involving partitions. This enables us to prove a topological generation result in Theorem \ref{thm:topological_generation}, which can be applied to the question of residual finiteness. We also provide a generalization of a result of F. Lemeux and P. Tarrago \cite{lemeux2014free} in Proposition \ref{prop:monoidal_equivalence} concerning monoidal equivalence, from which results on approximation properties can be deduced.

Eventually, in Section \ref{sec:extension} we show that, in the discrete case, when $\Lambda$ is a normal subgroup of $\Gamma$, the free wreath product of the pair fits into an exact sequence of Hopf $*$-algebra. We then give applications to cohomological dimension and determine when the exact sequence splits to yield a smash product decomposition.

\subsection*{Acknowledgement}

The author whishes to thank Julien Bichon for reading and commenting a preliminary version of this text, as well as for bibliographical assistance. He was supported by the ANR grants ``Noncommutative analysis on groups and quantum groups'' (ANR-19-CE40-0002) and ``Operator algebras and dynamics on groups'' (ANR-19-CE40-0008).

\section{Preliminaries}

In this preliminary section we will recall some basic facts about compact quantum groups and their combinatorics, mainly to fix terminology and notations. We refer the reader to \cite{timmermann2008invitation} and \cite{neshveyev2014compact} for detailed treatments of the theory. It is known since the work of M. Dijkhuizen and T. Koornwinder \cite{dijkhuizen1994CQG} that compact quantum groups can be treated algebraically through the following notion of a CQG-algebra.

\begin{de}
A \emph{CQG-algebra} is a Hopf $*$-algebra which is spanned by the coefficients of its finite-dimensional unitary corepresentations.
\end{de}

If $G$ is a compact group, then its algebra of regular functions $\O(G)$ is a CQG-algebra. Based on that example, and in an attempt to retain the intuition coming from the classical setting, we will denote a general CQG-algebra by $\O(\G)$ and say that it corresponds to the \emph{compact quantum group} $\G$. If $\Gamma$ is a discrete group and $\C[\Gamma]$ denotes its group algebra, it is easy to endow it with a Hopf $*$-algebra structure with coproduct given by $\D(g) = g\otimes g$ for all $g\in \Gamma$. Since this turns each $g\in \Gamma\subset\C[\Gamma]$ into a one-dimensional co-representation, it yields a CQG-algebra. The resulting compact quantum group is called the \emph{dual} of $\Gamma$ and is denoted by $\widehat{\Gamma}$.

The main feature of compact quantum groups is their nice representation theory, which is just another point of view on the corepresentation theory of the corresponding CQG-algebra. Let us give a definition to make this clear.

\begin{de}
An \emph{$n$-dimensional representation} of a compact quantum group $\G$ is an element $v\in M_{n}(\O(\G))$ which is invertible and such that for all $1\leqslant i, j\leqslant n$,
\begin{equation*}
\D(v_{ij}) = \sum_{k=1}^{n}v_{ik}\otimes v_{kj}.
\end{equation*}
It is said to be \emph{unitary} if it is unitary as an element of $M_{n}(\O(\G))$.
\end{de}

Given two representations $v$ and $w$, one can form their direct sum by considering a block diagonal matrix with blocks $v$ and $w$ respectively, and their tensor product by considering the matrix with coefficients
\begin{equation*}
(v\otimes w)_{(i,k),(j,\ell)} = v_{ij}w_{k\ell}.
\end{equation*}
In this setting, an intertwiner between two representations $v$ and $w$ of dimension respectively $n$ and $m$ will be a linear map $T : \C^{n}\to \C^{m}$ such that $Tv = wT$ (we are here identifying $M_{n}(\C)$ with $M_{n}(\C.1_{\O(\G)})\subset M_{n}(\O(\G))$). The set of all intertwiners between $v$ and $w$ will be denoted by $\displaystyle\Mor_{\G}(v, w)$.

If $T$ is injective, then $v$ is said to be a \emph{subrepresentation} of $w$, and if $w$ admits no subrepresentation apart from itself, then it is said to be \emph{irreducible}. One of the fundamental results in the representation theory of compact quantum groups is due to S.L. Woronowicz in \cite{woronowicz1995compact} and can be summarized as follows :

\begin{thm}[Woronowicz]
Any finite-dimensional representation of a compact quantum group splits as a direct sum of irreducible ones, and any irreducible representation is equivalent to a unitary one.
\end{thm}

We will need to consider subgroups in this quantum setting, and this can be done in two ways. First, if $G$ is a compact group and $H$ is a closed subgroup, then the restriction of functions yields a Hopf algebra $*$-homomorphism $\pi : \O(G)\to \O(H)$, leading to the following definition :

\begin{de}
Let $\G$ and $\HH$ be compact quantum groups. We say that $\HH$ is a \emph{quantum subgroup} of $\G$ if there exists a surjective Hopf algebra $*$-homomorphism $\pi : \O(\G)\to \O(\HH)$. We then write $\HH < \G$.
\end{de}

Consider now the dual of a discrete group $\Gamma$ and let $A\subset \O(\widehat{\Gamma}) = \C[\Gamma]$ be a Hopf $*$-subalgebra. Then, there exists a subgroup $\Lambda$ of $\Gamma$ such that $A = \C[\Lambda]$. Abstracting this yields

\begin{de}
Let $\G$ and $\HH$ be compact quantum groups. If $\O(\HH)$ is a Hopf $*$-subalgebra of $\O(\G)$, then $\HH$ is said to be a \emph{dual quantum subgroup} of $\G$.
\end{de}

\begin{rem}
With a theory of discrete quantum groups at hand together with the corresponding generalization of Pontryagin duality (which we will not introduce because we do not need it later on), the previous definition becomes very natural since it corresponds to the discrete dual $\widehat{\HH}$ of $\HH$ being a quantum subgroup of the discrete dual $\widehat{\G}$ of $\G$. One could therefore write this as $\widehat{\HH} < \widehat{\G}$ but we will avoid that notation in the sequel.
\end{rem}

Part of the present work relies on the combinatorial approach to compact quantum groups initally developped by T. Banica and R. Speicher in \cite{banica2009liberation}. We will use here the coloured version of that theory introduced in \cite{freslon2014partition}, which we very briefly recall. The basic objects are partitions of finite sets, which we represent by drawing the elements of the set on two parallel rows and connecting them by lines if they belong to the same component of the partition. Such a graphical interpretation enables to concatenate partitions horizontally and vertically (if the number of points matches) and rotate points from one line to another.

A \emph{category of partitions} is a collection of partitions $\CC(k, \ell)$ on $k+\ell$ points for all integers $k, \ell$ which is globally stable under the above operations and contains the identity partition $\vert$. It is moreover said to be \emph{coloured} if to each point is attached an element of a colour set $\A$, which is simply a set endowed with an involution $x\mapsto x^{-1}$. When a point of a partition is rotated from one row to another, its colour is changed into its image under the involution. Moreover, vertical concatenation of coloured partitions is only allowed when the colours of the points which are merged match. We also have a coloured version of the identity partition for each $x\in \A$, called the $x$-identity partition. To each coloured partition $p$ we associate the word $w$ formed by the colours of the upper row (from left to right) and the word $w'$ formed by the colours of its lower row (again from left to right). We then write $p\in \CC(w, w')$.

The framework of easy quantum groups generalizes to this setting. In particular, to any category of coloured partitions $\CC$ and any integer $N$ is associated a compact quantum group $\G_{N}(\CC)$. Its CQG-algebra is generated by the coefficients of unitary representations $(u^{x})_{x\in \A}$ indexed by the colours and any irreducible representation is a subrepresentation of a tensor product
\begin{equation*}
u^{\otimes w} = u^{w_{1}}\otimes\cdots\otimes u^{w_{n}},
\end{equation*}
where $w = w_{1}\cdots w_{n}$ is a word over $\A$. Moreover, the intertwiner spaces are completely determined by the category of partitions, in the sense that
\begin{equation*}
\Mor_{\G_{N}(\CC)}\left(u^{\otimes w}, u^{\otimes w'}\right) = \Span\left\{T_{p} \mid p\in \CC(w, w')\right\}.
\end{equation*}
Here, $T_{p}$ is a linear map the definition of which we now recall. Writing indices $i_{1}, \cdots, i_{n}$ from left to right on the upper row of $p$ and indices $j_{1}, \cdots, j_{k}$ from left to right on the lower row, we set $\delta_{p}(i_{1} \cdots i_{n}, j_{1}\cdots j_{k}) = 1$ if all indices which are connected by strings of $p$ are equal, and $\delta_{p}(i_{1} \cdots i_{n}, j_{1}\cdots j_{k}) = 0$ otherwise. Then, if $(e_{i})_{1\leqslant i\leqslant N}$ is the canonical basis of $\C^{N}$, we set
\begin{equation*}
T_{p}(e_{i_{1}}\otimes\cdots\otimes e_{n}) = \sum_{j_{1}, \cdots, j_{k}=1}^{N}\delta_{p}(i_{1} \cdots i_{n}, j_{1}\cdots j_{k}) e_{j_{1}}\otimes\cdots\otimes e_{j_{k}}.
\end{equation*}

\section{General theory}\label{sec:general}

\subsection{Definition and first examples}

Free wreath products of pairs were introduced in \cite{freslon2017noncrossing} through the use of categories of coloured partitions. The definition can be seen as a generalization of the categories of partitions of free wreath products computed in \cite{lemeux2013haagerup}, hence the name. However, free wreath products were defined in purely Hopf algebraic terms by J. Bichon in \cite{bichon2004free}. We will therefore start by providing a definition in the same spirit. At this level, working with arbitrary compact quantum groups does not entail any additional difficulty, hence we will use that setting for the moment.

Free wreath products can be defined as quotients of a free product involving a compact quantum group $\G$ and the quantum permutation group $S_{N}^{+}$ introduced by S. Wang in \cite{wang1998quantum}, whose definition we now recall.

\begin{de}
The CQG-algebra of the quantum permutation group $S_{N}^{+}$ is the universal unital $*$-algebra generated by elements $(p_{ij})_{1\leqslant i, j\leqslant N}$ such that
\begin{itemize}
\item $p_{ij}^{2} = p_{ij} = p_{ij}^{*}$ for all $1\leqslant i, j\leqslant N$ ;
\item $\displaystyle\sum_{k=1}^{N}p_{ik} = 1 = \displaystyle\sum_{k=1}^{N}p_{kj}$ for all $1\leqslant i, j\leqslant N$ ;
\item $p_{ik}p_{jk} = \delta_{ij}p_{ik}$ and $p_{ki}p_{kj} = \delta_{ij}p_{kj}$ for all $1\leqslant i, j\leqslant N$ ;
\end{itemize}
endowed with the unique coproduct $\D : \O(S_{N}^{+})\to \O(S_{N}^{+})\otimes \O(S_{N}^{+})$ such that for all $1\leqslant i, j\leqslant N$,
\begin{equation*}
\D(p_{ij}) = \sum_{k=1}^{N}p_{ik}\otimes p_{kj}
\end{equation*}
and structure maps given by $\varepsilon(p_{ij}) = \delta_{ij}$ and $S(p_{ij}) = p_{ji}$.
\end{de}

We will follow the same path, and the novelty of the free wreath product with amalgamation construction can indeed be seen at this level simply as an additional amalgamation introduced in the free product. More precisely, we will consider the following $*$-algebra : 

\begin{de}\label{de:commutationrelations}
Let $\G$ be a compact quantum group and let $\HH$ be a dual quantum subgroup of $\G$. Let
\begin{equation*}
A = \O(\G)^{\ast_{\O(\HH)}^{N}}
\end{equation*}
be the iterated amalgamated free product and denote by $\nu_{i} : \O(\G)\to A$ the canonical inclusion into the $i$-th factor. Then, the $*$-algebra $\O(\G\wr_{\ast, \HH}S_{N}^{+})$ is defined as the quotient of $A$ by the relations
\begin{equation*}
[\nu_{i}(x), p_{ij}] = 0
\end{equation*}
for all $1\leqslant i, j\leqslant N$ and all $x\in \O(\G)$.
\end{de}

Note that the CQG-algebra of the usual free wreath product $\G\wr_{\ast}S_{N}^{+}$ can be recovered as the special case where $\HH = \{e\}$ is the trivial group (i.e. the inclusion $\C.1_{\O(\G)}\subset \O(\G)$). Moreover, we can in general see $\O(\G\wr_{\ast, \HH}S_{N}^{+})$ as the quotient of $\O(\G\wr_{\ast}S_{N}^{+})$ by the ideal generated by $\nu_{i}(x) - \nu_{j}(x)$ for all $1\leqslant i\neq j\leqslant N$ and $x\in \O(\HH)$.

The main technical point in the definition of free wreath products in \cite{bichon2004free} is the coproduct, whose formula is not completely obvious. It is proven in \cite[Thm 2.3]{bichon2004free} that, denoting by $\D_{\G}$ the coproduct on $\G$, there is a unique coproduct $\D$ on $\O(\G\wr_{\ast}S_{N}^{+})$ such that
\begin{equation*}
\D(p_{ij}) = \sum_{k=1}^{N}p_{ik}\otimes p_{kj} \quad \& \quad \D(\nu_{i}(x)) = \sum_{k=1}^{N}(\nu_{i}\otimes \nu_{k})(\D_{\G}(x))(p_{ik}\otimes 1)
\end{equation*}
for all $1\leqslant i, j\leqslant N$ and all $x\in \O(\G)$. The most natural thing to do is therefore to check that these formul\ae{} again define a coproduct on $\O(\G\wr_{\ast, \HH}S_{N}^{+})$.

\begin{prop}
Let $\G$ be a compact quantum group and let $\HH$ be a dual quantum subgroup of $\G$. Then, the previous coproduct on $\O(\G\wr_{\ast}S_{N}^{+})$ factors through the quotient onto $\O(\G\wr_{\ast, \HH}S_{N}^{+})$, endowing the latter with a compact quantum group structure.
\end{prop}

\begin{proof}
All that we have to do is to check that the $*$-ideal generated by the elements $\nu_{i}(x) - \nu_{j}(x)$ for all $x\in\O(\HH)$ is a Hopf ideal. Indeed, using Sweedler's notation to write $\D_{\G}(x) = x_{1}\otimes x_{2}\in \O(\HH)\otimes \O(\HH)$,
\begin{align*}
\D(\nu_{i}(x) - \nu_{j}(x)) & = \sum_{k=1}^{N}(\nu_{i}\otimes \nu_{k})(x_{1}\otimes x_{2})(p_{ik}\otimes 1) - (\nu_{j}\otimes \nu_{k})(x_{1}\otimes x_{2})(p_{jk}\otimes 1) \\
& = \sum_{k, k'=1}^{N}\nu_{i}(x_{1})p_{ik}\otimes (\nu_{k} - \nu_{k'})(x_{2}) + \sum_{k, k'=1}^{N}\nu_{i}(x_{1})p_{ik}\otimes \nu_{k'}(x_{2}) \\
& - \sum_{k, k'=1}^{N}\nu_{j}(x_{1})p_{jk}\otimes (\nu_{k} - \nu_{k'})(x_{2}) - \sum_{k, k'=1}^{N}\nu_{j}(x_{1})p_{jk}\otimes \nu_{k'}(x_{2}) \\
& = \sum_{k, k'=1}^{N}\nu_{i}(x_{1})p_{ik}\otimes (\nu_{k} - \nu_{k'})(x_{2}) + \sum_{k'=1}^{N}\nu_{i}(x_{1})\otimes \nu_{k'}(x_{2}) \\
& - \sum_{k, k'=1}^{N}\nu_{j}(x_{1})p_{jk}\otimes (\nu_{k} - \nu_{k'})(x_{2}) - \sum_{k'=1}^{N}\nu_{j}(x_{1})\otimes \nu_{k'}(x_{2}) \\
& = \sum_{k, k'=1}^{N}\nu_{i}(x_{1})p_{ik}\otimes (\nu_{k} - \nu_{k'})(x_{2}) - \sum_{k, k'=1}^{N}\nu_{j}(x_{1})p_{jk}\otimes (\nu_{k} - \nu_{k'})(a_{2}) \\
& + \sum_{k'=1}^{N}(\nu_{i}(x_{1}) - \nu_{j}(x_{1})\otimes \nu_{k'}(x_{2}) \\
\end{align*}
proving our claim.

The fact that this yields a compact quantum group is then proven exactly as in the proof of \cite[Thm 2.3]{bichon2004free}.
\end{proof}

We are now ready for the definition of the main object of this work.

\begin{de}\label{de:freewreathproductwithamalgamation}
The compact quantum group associated to the pair $(\O(\G\wr_{\ast, \HH}S_{N}^{+}), \D)$ is called the \emph{free wreath product of $\G$ by $S_{N}^{+}$ with amalgamation over $\HH$} and is denoted by $\G\wr_{\ast, \HH}S_{N}^{+}$.
\end{de}

In order to illustrate the construction, let us show that it interpolates in a sense between the direct product and the free wreath product.

\begin{prop}\label{prop:directproduct}
Let $\G$ be a compact quantum group and let $\HH$ be a dual quantum subgroup. Then, for all $N\in\N$ there are inclusions
\begin{equation*}
\G\times S_{N}^{+} < \G\wr_{\ast, \HH}S_{N}^{+} < \G\wr_{\ast}S_{N}^{+}.
\end{equation*}
\end{prop}

\begin{proof}
Note first that if $\K$ is a dual quantum subgroup of $\HH$, then there is a canonical surjection
\begin{equation*}
\O(\G)^{\ast_{\O(\K)}^{N}}\to \O(\G)^{\ast_{\O(\HH)}^{N}}
\end{equation*}
since $\{\nu_{i}(x) - \nu_{j}(x) \mid x\in \O(\K)\}\subset \{\nu_{i}(x) - \nu_{j}(x) \mid x\in \O(\HH)\}$. Further quotienting by the commutation relations of Definition \ref{de:commutationrelations}, yields a surjective $*$-homomorphisms 
\begin{equation*}
\O(\G\wr_{\ast, \K}S_{N}^{+}) \to \O(\G\wr_{\ast, \HH}S_{N}^{+})
\end{equation*}
and because the definition of the coproduct is given by the same formul\ae{} of both sides, this yields an inclusion of compact quantum groups $\G\wr_{\ast, \HH}S_{N}^{+} < \G\wr_{\ast, \K}S_{N}^{+}$. Applying this to $\K=\G$ and $\K=\{e\}$ yields
\begin{equation*}
\G\wr_{\ast, \G}S_{N}^{+} < \G\wr_{\ast, \HH}S_{N}^{+} < \G\wr_{\ast, \{e\}}S_{N}^{+} = \G\wr_{\ast}S_{N}^{+}.
\end{equation*}
To conclude, is suffices to check that $\G\wr_{\ast, \G}S_{N}^{+} = \G\times S_{N}^{+}$. Indeed, $\O(\G)^{\ast_{\O(\G)}^{N}} = \O(\G)$ and the relations make that unique copy of $\O(\G)$ commute with all the generators of $\O(S_{N}^{+})$.
\end{proof}

\subsection{The group dual case}

We will now restrict to the case where $\G$ is the dual of a discrete group $\Gamma$. Then, $\HH$ is nothing but the dual of a subgroup $\Lambda$ of $\Gamma$. In that case, there is an alternative description of the free wreath product using explicit generators. Let us first recall how this works for free wreath products. Consider the universal $*$-algebra $\O(H_{N}^{+}(\Gamma))$ generated by elements $a_{ij}(g)$ for all $1\leqslant i, j\leqslant N$ and $g\in \Gamma$ subject to the relations.
\begin{align*}
a_{ij}(g)a_{ik}(h) & = \delta_{ik}a_{ij}(gh) \\
a_{ji}(g)a_{ki}(h) & = \delta_{jk}a_{ji}(gh) \\
\sum_{k=1}^{N}a_{ik}(e) & = 1 = \sum_{k=1}^{N}a_{kj}(e)
\end{align*}
Using the universal property, it is easy to see that there exists a coproduct $\D$ satisfying
\begin{equation*}
\D(a_{ij}(g)) = \sum_{k=1}^{N}a_{ik}(g)\otimes a_{kj}(g)
\end{equation*}
and yielding a compact quantum group structure on $\O(H_{N}^{+}(\Gamma))$. Moreover, it is proven in \cite[Ex 2.5]{bichon2004free} that the corresponding quantum group $H_{N}^{+}(\Gamma)$ is isomorphic to the free wreath product $\widehat{\Gamma}\wr_{\ast}S_{N}^{+}$.

We will now give an analogue of this result for pairs of discrete groups. Here is how the definition has to be modified.

\begin{de}
Let $\Gamma$ be a discrete group and let $\Lambda$ be a subgroup of $\Gamma$. The universal $*$-algebra $\O(H_{N}^{+}(\Gamma, \Lambda))$ is defined as the quotient of $\O(H_{N}^{+}(\Gamma))$ by the Hopf $*$-ideal generated by the relations
\begin{equation*}
\sum_{k=1}^{N}a(g)_{ik} = \sum_{k=1}^{N}a(g)_{jk}
\end{equation*}
for all $1\leqslant i, j \leqslant N$ and all $g\in \Lambda$.
\end{de}

\begin{rem}
Note that setting $s(g)$ to be the sum appearing in the definition (if it is independent of $i$), the set of elements $g$ such that $s(g)$ is well-defined must be a group : $s(e) = 1$, $s(g^{-1}) = S(s(g))$ (where $S$ denotes the antipode) and $s(gh) = s(g)s(h)$. Indeed,
\begin{align*}
s(g)\sum_{k=1}^{N}a_{ik}(h) & = \sum_{k, k' = 1}^{N}a_{ik'}(g)a_{ik}(h) \\
& = \sum_{k, k'=1}^{N}\delta_{k, k'}a_{ik}(gh) \\
& = \sum_{k=1}^{N} a_{ik}(gh)
\end{align*}
so that the sum in the last line does not depend on $i$.
\end{rem}

It is not difficult to check that the coproduct on $\O(H_{N}^{+}(\Gamma))$ again defines a coproduct on the quotient and that this yields a compact quantum group. This will however be a consequence of the next result.

\begin{prop}\label{prop:universaldiscrete}
There exists a $*$-isomorphism
\begin{equation*}
\Phi : \O(\widehat{\Gamma}\wr_{\ast, \widehat{\Lambda}}S_{N}^{+})\to \O(H_{N}^{+}(\Gamma, \Lambda))
\end{equation*}
such that push-forward of the coproduct maps $a_{ij}(g)$ to
\begin{equation*}
\sum_{k=1}^{N} a_{ik}(g)\otimes a_{kj}(g).
\end{equation*}
\end{prop}

\begin{proof}
As in \cite[Ex 2.5]{bichon2004free}, set
\begin{equation*}
\phi_{i}(g) = \sum_{k=1}^{N}a_{ik}(g) \quad \& \quad \phi_{N+1}(p_{ij}) = a_{ij}(e)
\end{equation*}
to define by universality a $*$-homomorphism
\begin{equation*}
\C[\Gamma]^{\ast N}\ast S_{N}^{+} \to \O(H_{N}^{+}(\Gamma, \Lambda)).
\end{equation*}
Observing that $\phi_{i}(g) = \phi_{j}(g)$ for all $1\leqslant i, j \leqslant N$ and all $g\in \Lambda$, we see that this map factors to a $*$-homomorphism
\begin{equation*}
\phi : \O(\widehat{\Gamma}\wr_{\ast, \widehat{\Lambda}} S_{N}^{+}) \to \O(H_{N}^{+}(\Gamma, \Lambda)).
\end{equation*}
Conversely, one easily checks that the elements
\begin{equation*}
\psi_{ij}(g) = \nu_{i}(g)p_{ij}\in \O(\widehat{\Gamma}\wr_{\ast, \widehat{\Lambda}}S_{N}^{+}))
\end{equation*}
satisfy the defining relations of $\O(H_{N}^{+}(\Gamma))$. Moreover, for $g\in \Lambda$ and $1\leqslant i, j\leqslant N$,
\begin{equation*}
\sum_{k=1}^{N}\psi_{ik}(g) = \nu_{i}(g) = \nu_{j}(g) = \sum_{k=1}^{N}\psi_{jk}(g)
\end{equation*}
so that these elements even satisfy the defining relations of $\O(H_{N}^{+}(\Gamma, \Lambda))$. By universality, this induces a $*$-homomorphism 
\begin{equation*}
\psi : \O(H_{N}^{+}(\Gamma, \Lambda)) \to \O(\widehat{\Gamma}\wr_{\ast, \widehat{\Lambda}}S_{N}^{+})
\end{equation*}
which a straightforward computation shows to be inverse to $\phi$. The formula for the coproduct is then obtained by conjugating by these $*$-homomorphisms.
\end{proof}

When $\Gamma$ (hence also $\Lambda$) is abelian, another description of the amalgamated free wreath product is possible, involving the notion of \emph{glued direct product} from \cite{gromada2020gluing} in a specific case. Because $\Gamma$ is assumed to be abelian, $\Lambda$ is normal and we can consider the usual free wreath product $H_{N}^{+}(\Gamma/\Lambda) = H_{N}^{+}(\Gamma/\Lambda, \{e\})$ and build its direct product in the sense of compact quantum groups with $\widehat{\Gamma}$. At the level of CQG-algebras, this yields
\begin{equation*}
B = \O(H_{N}^{+}(\Gamma/\Lambda))\otimes \O(\widehat{\Gamma}).
\end{equation*}
Now, we denote by $\O(H_{N}^{+}(\Gamma/\Lambda)\widetilde{\times}\widehat{\Gamma})$ the $*$-subalgebra of $B$ generated by the elements $a(g)_{ij}g$ for all $g\in \Gamma$ and $1\leqslant i, j\leqslant N$. This is a CQG-algebra for the restriction of the coproduct, and the corresponding compact quantum group is called the glued direct product of $H_{N}^{+}(\Gamma/\Lambda)$ and $\widehat{\Gamma}$.

\begin{rem}
The previous construction is slightly outside the setting of \cite{gromada2020gluing}. However, it can easily be made to fit in if $\Gamma$ is finitely generated. In that case, if $\SS$ is a finite generating set, then one can consider $H_{N}^{+}(\Gamma/\Lambda)$ and $\widehat{\Gamma}$ as compact matrix quantum groups with fundamental representations respectively $\oplus_{g\in \SS}a(g)$ and $\oplus_{g\in \SS}g$, and the glued product construction then coincides with ours.
\end{rem}

The proof of the next theorem will make use of an explicit description of the representation theory of $H_{N}^{+}(\Gamma, \Lambda)$ obtained in \cite[Thm 4.4]{freslon2018representation}. Strictly speaking, this result concerns the partition quantum groups $\G_{N}(\CC_{\Gamma, \Lambda, S})$, but we will prove that it coincides with $H_{N}^{+}(\Gamma, \Lambda)$ in Proposition \ref{prop:partitiondescription}. Let $F(\Gamma)$ be the free monoid on $\Gamma$ and consider the relation $\sim$ defined on it by
\begin{equation*}
w_{1}\dots (w_{i}.\lambda)w_{i+1}\dots w_{n}\sim w_{1}\dots w_{i}(\lambda.w_{i+1})\dots w_{n}
\end{equation*}
for any $1\leqslant i\leqslant n-1$ and $\lambda\in \Lambda$. The quotient by the transitive closure of the relation will be denoted by $W(\Gamma, \Lambda)$. Then, the one-dimensional representations of $H_{N}^{+}(\Gamma, \Lambda)$ can be indexed by the elements of $\Lambda$ and all the other ones by elements of $W(\Gamma, \Lambda)$. Moreover, if $g_{1}\cdots g_{n}\in W(\Gamma, \Lambda)$ and $\lambda\in \Lambda$, then $\lambda\otimes (g_{1}\cdots g_{n}) = (\lambda g_{1})g_{2}\cdots g_{n}$. We are now ready for the result.

\begin{thm}\label{thm:gluedproduct}
If $\Gamma$ is abelian, then there is an isomorphism of compact quantum groups
\begin{equation*}
H_{N}^{+}(\Gamma, \Lambda) \cong H_{N}^{+}(\Gamma/\Lambda)\widetilde{\times}\widehat{\Gamma}.
\end{equation*}
\end{thm}

\begin{proof}
Let us set, for $g\in \Gamma$, $b(g) = a([g])g\in \O(H_{N}^{+}(\Gamma/\Lambda)\widetilde{\times}\widehat{\Gamma})$, where $[g]$ denotes the image of $g$ in $\Gamma/\Lambda$. Then, for any $1\leqslant i, j\leqslant N$,
\begin{align*}
b(g)_{ij}b(h)_{ik} & = a([g])_{ij}ga([h])_{ik}h \\
& = a([g])_{ij}a([h])_{ik}gh \\
& = \delta_{jk}a([gh])_{ij}gh \\
& = \delta_{jk}b(gh)_{ij}
\end{align*}
and similarly $b(g)_{ji}b(g)_{ki} = \delta_{jk}b(gh)_{ji}$. Moreover, for $g\in \Lambda$ and $1\leqslant i\leqslant N$,
\begin{equation*}
\sum_{k=1}^{N}b(g)_{ik} = \left(\sum_{k=1}^{N}a([g])_{ik}\right)g = \left(\sum_{k=1}^{N}a([e])_{ik}\right)g = g.
\end{equation*}
As a consequence, there is a surjective $*$-homomorphism
\begin{equation*}
\Phi : \O(H_{N}^{+}(\Gamma, \Lambda)) \to \O(H_{N}^{+}(\Gamma/\Lambda)\widetilde{\times}\widehat{\Gamma})
\end{equation*}
sending $a(g)$ to $b(g)$.

To prove injectivity, it is enough to prove that the image of any irreducible representation under $\Phi$ is irreducible. Let us first prove by induction on $n$ that if $g_{1}, \cdots, g_{n}$ are elements of $\Gamma$ and $\lambda_{1}, \cdots, \lambda_{n}$ are elements of $\Lambda$, then
\begin{equation*}
\Phi(u^{g_{1}}\lambda_{1}\cdots u^{g_{n}}\lambda_{n}) = \lambda_{1}\cdots \lambda_{n}u^{[g_{1}]}\cdots u^{[g_{n}]}g_{1}\cdots g_{n}.
\end{equation*}
For $n = 1$, the result is clear. Assume that it holds for some $n$. If $g_{n}g_{n+1}\notin \Lambda$, then the fusion rules computed in \cite[Prop 4.6]{freslon2018representation} imply that
\begin{align*}
\Phi(u^{g_{1}}\lambda_{1}\cdots u^{g_{n}}\lambda_{n}u^{g_{n+1}}\lambda_{n+1}) & = \Phi(\lambda_{1}\cdots\lambda_{n+1}u^{g_{1}}\cdots u^{g_{n}}\otimes u^{g_{n+1}}) \\
& = \lambda_{1}\cdots\lambda_{n+1}u^{[g_{1}]}\cdots u^{[g_{n}]}\otimes u^{[g_{n+1}]}g_{1}\cdots g_{n+1} \\
& = \lambda_{1}\cdots\lambda_{n+1}u^{[g_{1}]}\cdots u^{[g_{n}]}u^{[g_{n+1}]}g_{1}\cdots g_{n+1} \\
\end{align*}
Otherwise,
\begin{align*}
\Phi(u^{g_{1}}\lambda_{1}\cdots u^{g_{n}}\lambda_{n}g_{n+1}\lambda_{n+1}) & = \Phi\left((\lambda_{1}\cdots\lambda_{n+1}\left(u^{g_{1}}\cdots u^{g_{n}}\otimes u^{g_{n+1}} - u^{g_{1}}\cdots u^{g_{n}g_{n+1}}\right)\right)) \\
& = \lambda_{1}\cdots\lambda_{n+1}\left(u^{[g_{1}]}\cdots u^{[g_{n}]}\otimes u^{[g_{n+1}]} - u^{[g_{1}]}\cdots u^{[g_{n}g_{n+1}]}\right) g_{1}\cdots g_{n+1} \\
& = \lambda_{1}\cdots\lambda_{n+1}u^{[g_{1}]}\cdots u^{[g_{n}]}u^{[g_{n+1}]}g_{1}\cdots g_{n+1} \\
\end{align*}
so that our claim is proven.

To conclude we have to check that the images are irreducible representations. Note that the irreducible representations of $H_{N}^{+}(\Gamma/\Lambda)\times \widehat{\Gamma}$ are, by \cite[Thm 2.11]{wang1995tensor}, of the form $gu^{[g_{1}]}\cdots u^{[g_{n}]}$ for $g, g_{1}, \cdots, g_{n}\in \Gamma$. Moreover, $\O(H_{N}^{+}(\Gamma/\Lambda)\widetilde{\times}\widehat{\Gamma})$ being a Hopf $*$-subalgebra of $\O(H_{N}^{+}(\Gamma/\Lambda)\times \widehat{\Gamma})$, its irreducible representations are exactly those of $H_{N}^{+}(\Gamma/\Lambda)\times \widehat{\Gamma}$ whose coefficients all lie in $\O(H_{N}^{+}(\Gamma/\Lambda)\widetilde{\times}\widehat{\Gamma})$. In other words, the images under $\Phi$ of irreducible representations are irreducible, hence $\Phi$ is an isomorphism.
\end{proof}

Applying Theorem \ref{thm:gluedproduct} to cyclic groups, we deduce that if $d\mid k$, then
\begin{equation*}
\left(\Z_{d}\wr_{\ast}S_{N}^{+}\right)\widetilde{\times}\Z_{k}\cong \Z_{k}\wr_{\ast, d\Z_{k}}S_{N}^{+},
\end{equation*}
thereby recovering a result by P. Tarrago and M. Weber \cite{tarrago2015unitary}.

\begin{rem}
In the case of cyclic groups, the glued product was introduced earlier in \cite{tarrago2015unitary} under the name of \emph{$k$-tensor complexification}. These complexification were later studied by D. Gromada and M. Weber in \cite{gromada2019new} were they consider several variants of them and consider them as \emph{extensions} of a given compact quantum group by $\Z_{2}$. The term ``extension'' was used there in a broad sense, but Proposition \ref{prop:exactsequence} will show that in our case, the glued direct product yields a genuine extension in the sense of Hopf algebras. It is nevertheless important to note that the extension does not involve $\Z_{k}$ itself but rather the subgroup $d\Z_{k}$.
\end{rem}

\subsection{The classical construction}

As is usual with such constructions in compact quantum group theory, the free wreath product with amalgamation has a classical counterpart obtained by considering the abelianization of the CQG-algebra. To describe this in more details, first note that if $G$ is a classical compact group, then a dual quantum subgroup is simply given by a quotient compact group $H$. 

\begin{de}
Let $G$ be a compact group and $\pi : G\to H$ a quotient map. The \emph{wreath product of $G$ by $S_{N}$ with amalgamation over $H$} is defined to be the classical compact group corresponding to the CQG-algebra $\O(G\wr_{\ast, H}S_{N}^{+})_{\text{ab}}$, where the index ``ab'' denotes abelianization.
\end{de}

In order to give alternate descriptions of this object, let us recall some elementary facts. First, the abelianization of the amalgamated free product $\O(G)^{\ast_{\O(H)}^{N}}$ is the amalgamated direct product
\begin{equation*}
\O(G)^{\times_{\O(H)}^{N}} = \O(G^{\times_{H}^{N}}).
\end{equation*}
Second, the amalgamated direct product $G^{\times_{H}^{N}}$ is the subgroup of $G^{\times N}$ of elements whose image in $H^{\times N}$ is diagonal. Third, the set of ``diagonal'' elements in $H\wr S_{N}$ is a subgroup isomorphic to $H\times S_{N}$.

\begin{prop}
Let $G$ be a compact group and let $\pi : G\to H$ be a quotient map. The following compact groups are isomorphic to the wreath product of $G$ by $S_{N}$ with amalgamation over $H$ :
\begin{enumerate}
\item $G^{\times^{N}_{H}}\rtimes S_{N}$, where $S_{N}$ acts by permutation of the copies ;
\item The pre-image of $H\times S_{N}\subset H\wr S_{N}$ under the surjective group homomorphism $G\wr S_{N}\to H\wr S_{N}$ induced by $\pi$.
\end{enumerate}
\end{prop}

\begin{proof}
\begin{enumerate}
\item The abelianization of $\O(G\wr_{\ast}S_{N}^{+})$ is $\O(G\wr S_{N})$, which is the CQG-algebra corresponding to the group $G^{\times N}\rtimes S_{N}$ with $S_{N}$ acting by permutation of the copies. The result therefore follows from the fact that $\O(G\wr_{\ast, H} S_{N}^{+})_{\text{ab}}$ is the quotient of $\O(G\wr S_{N})$ by the relations identifying the copies of $C(H)\subset C(G)$.
\item Recall that $\O(H)$ is seen as a subalgebra of $\O(G)$ through the map $f\mapsto f\circ \pi$. As a consequence, the amalgamated tensor product of $\O(G)$ over $\O(H)$ consists in quotienting $\O(G)^{\otimes N}$ by the relations $\pi(x_{i}) = \pi(x_{j})$ for all $x_{1}\otimes \cdots\otimes x_{N}\in \O(G)^{\otimes N}$ and all $1\leqslant i, j\leqslant N$. At the level of groups, this is equivalent to restricting to the subgroup of tuples of elements of $G$ having the same image in $H$, hence the result.
\end{enumerate}
\end{proof}

Using this result, we can give yet another characterization in the cyclic case. Recall that a matrix is called \emph{monomial} if it has exactly one non-zero coefficient in each row and column. The reflection group $H_{N}^{k} = H_{N}(\Z_{k})$ can then be described as the group of all monomial matrices in $M_{N}(\C)$ whose coefficients are $k$-th roots of unity.

\begin{prop}
Let $d, k\in \N$ with $d\mid k$. Then, the group $\Z_{k}\wr_{\Z_{d}}S_{N}$ is the group of all monomial matrices in $H_{N}^{k}$ such that all non-zero coefficients have equal $k/d$-th power.
\end{prop}

\begin{proof}
Set $\omega = e^{2i\pi/k}$ so that $\Z_{d}$ is seen as the subgroup generated by $\omega^{k/d}$. Let now $M\in H_{N}^{k}$ and let $(z_{1}, \cdots, z_{N})$ be its non-zero coefficients. Assume that $z_{i}^{d} = z_{1}^{d}$ for all $2\leqslant i\leqslant N$. Then, there exists $j_{1}, \cdots j_{N}$ such that $z_{i} = \omega^{kj_{i}/d}z_{1}$. In other words, $M$ is the product of a monomial matrix with coefficients $(\omega^{kj_{1}/d}, \cdots, \omega^{kj_{N}/d})$ and the diagonal matrix with constant coefficient $z_{1}$. This is an element of $H_{N}(\Z_{k/d})\widetilde{\times}\Z_{k}$, which by Theorem \ref{thm:gluedproduct} is isomorphic to $\Z_{k}\wr_{\Z_{d}}S_{N}$. We have therefore proven one inclusion, and the converse one is trivial.
\end{proof}

\section{Coloured partitions and applications}\label{sec:partitions}

We will now investigate the interplay between free wreath products with amalgamation for duals of discrete groups and coloured partitions in the sense of \cite{freslon2014partition}. This will in particular clarify the connection between this work and the original definition from \cite{freslon2017noncrossing}.

\subsection{Free wreath products with amalgamation as partition quantum groups}

As explained in the introduction, free wreath products with amalgamation originally appeared in \cite{freslon2017noncrossing}, under the name of free wreath products of pairs, as a specific family of partition quantum groups. They were therefore defined through a category of partitions, which is denoted by $\CC_{\Gamma, \Lambda, S}$. Let us recall that definition.

\begin{de}
Let $\Gamma$ be a discrete group with a symmetric generating set $\SS$ not containing the neutral element and let $\Lambda$ be a subgroup. We consider $\SS$ as a colour set with involution $g\to g^{-1}$. Then, $\CC_{\Gamma, \SS}$ is the category of all non-crossing $\SS$-coloured partitions such that in each block, the product of the colours in the upper row (from left to right) equals the product of the colours in the lower row (also from left to right) as elements of $\Gamma$.

If now $g\in \Lambda$ can be written as a product $g = g_{1}\cdots g_{n}$ of elements of $\SS$, we set
\begin{center}
\begin{tikzpicture}[scale=0.5]
\draw (-2.5,0.5) -- (2.5,0.5);
\draw (-2.5,-0.5) -- (2.5,-0.5);

\draw (-2.5,0.5) -- (-2.5,1.5);
\draw (2.5,0.5) -- (2.5,1.5);
\draw (-2.5,-0.5) -- (-2.5,-1.5);
\draw (2.5,-0.5) -- (2.5,-1.5);

\draw (-1,0.5) -- (-1,1.5);
\draw (1,0.5) -- (1,1.5);
\draw (-1,-0.5) -- (-1,-1.5);
\draw (1,-0.5) -- (1,-1.5);

\draw (0,1.5) node[below]{$\dots$};
\draw (0,-1.5) node[above]{$\dots$};

\draw (-2.5,1.5) node[above]{$g_{1}$};
\draw (-1,1.5) node[above]{$g_{2}$};
\draw (1,1.5) node[above]{$g_{n-1}$};
\draw (2.5,1.5) node[above]{$g_{n}$};

\draw (-2.5,-1.5) node[below]{$g_{1}$};
\draw (-1,-1.5) node[below]{$g_{2}$};
\draw (1,-1.5) node[below]{$g_{n-1}$};
\draw (2.5,-1.5) node[below]{$g_{n}$};

\draw (-3,0) node[left]{$\beta_{g} = $};
\end{tikzpicture}
\end{center}
and let $\CC_{\Gamma, \Lambda, \SS}$ be the category of partitions generated by $\CC_{\Gamma, \SS}$ and $\beta_{g}$ for all $g\in \lambda$.
\end{de}

We will now prove, using the universal algebra picture, that this defines the same object as the two previous constructions. The idea is simply that adding a partition to a category of partitions is equivalent to forcing a linear a map to intertwine two representations, hence amounts to adding a polynomial relation between the generators (see for instance \cite{freslon2017noncrossing} for details and illustrations of this principle).

\begin{prop}\label{prop:partitiondescription}
There is an isomorphism of compact quantum groups
\begin{equation*}
\G_{N}(\CC_{\Gamma, \Lambda, \SS}) \cong H_{N}^{+}(\Gamma, \Lambda).
\end{equation*}
\end{prop}

\begin{proof}
We already know from \cite{lemeux2013haagerup} that $\G_{N}(\CC_{\Gamma, \SS})\cong H_{N}^{+}(\Gamma)$ and that this isomorphism sends the representation corresponding to the $g$-identity partition to the representation $a(g) = (a(g)_{ij})_{1\leqslant i, j\leqslant N}$. Now, adding the partition $\beta_{g}$ is equivalent to adding the relations
\begin{equation*}
T_{\beta_{g}} a(g) = a(g) T_{\beta_{g}}.
\end{equation*}
But $T_{\beta_{g}}$ is, up to multiplication by a scalar, the orthogonal projection onto the vector $\xi = \sum e_{i}$. Thus, adding it to the category of partitions amounts to imposing conditions on $a(g)$ making $\xi$ fixed. Because
\begin{equation*}
a(g)\xi = \sum_{i, j=1}^{N}a_{ij}(g)\otimes e_{i},
\end{equation*}
$\xi$ is fixed if and only if the sum of the coefficients of $a(g)$ on each row is constant, which is exactly the relation defining $H_{N}^{+}(\Gamma, \Lambda)$.
\end{proof}

\subsection{Topological generation and monoidal equivalence}

We will now give two applications of the previous result to structural and approximation properties. Recall that given a compact quantum group $\G$ and a family of compact quantum subgroups $\O(\HH_{i})$ for $i$ ranging in some set $I$, we say that $\G$ is \emph{topologically generated} by these quantum subgroups if for any representations $u$ and $v$ of $\G$,
\begin{equation*}
\Mor_{\G}(u, v) = \bigcap_{i\in I}\Mor_{\HH_{i}}(u, v).
\end{equation*}
That property was first introduced by A. Chirvasitu in \cite{chirvasitu2015residually} to study residual finite-dimensionality of quantum group C*-algebras and we will exploit it here for the same purpose. Let us say that a compact quantum group $\G$ is \emph{residually finite} if its CQG-algebra $\O(\G)$ is residually finite-dimensional as a complex $*$-algebra, i.e. its finite-dimensional $*$-algebra representations separate the points.

Let $\Lambda < \Gamma$ be a pair of discrete groups. We have several natural quantum subgroups of $\widehat{\Gamma}\wr_{\ast, \widehat{\Lambda}}S_{N}^{+}$, and in particular we have $\widehat{\Gamma}\times S_{N}^{+} = \widehat{\Gamma}\wr_{\ast, \widehat{\Gamma}}S_{N}^{+}$ already mentioned in Proposition \ref{prop:directproduct}. Moreover, if $\overline{\Lambda}$ denotes the normal closure of $\Lambda$ (that is, the intersection of all normal subgroups of $\Gamma$ containing $\Lambda$), then $\widehat{\Gamma/\overline{\Lambda}}\wr_{\ast}S_{N}^{+}$ is a quantum subgroup of $\widehat{\Gamma}\wr_{\ast, \widehat{\Lambda}}S_{N}^{+}$. This in fact follows from the following result :

\begin{lem}
Let $\beta_{g}^{u}$ denote the upper part of the partition $\beta_{g}$ and set
\begin{equation*}
\CC = \left\langle \CC_{\Gamma, \Lambda, \SS}, \beta_{g}^{u} \mid g\in \Lambda\right\rangle.
\end{equation*}
Then,
\begin{equation*}
\G_{N}(\CC) = \widehat{\Gamma/\overline{\Lambda}}\wr_{\ast}S_{N}^{+}
\end{equation*}
\end{lem}

\begin{proof}
Because $\CC_{\Gamma, \Lambda, \SS}\subset \CC$, we know by the proof of \cite[Prop 3.17]{freslon2017noncrossing} that
\begin{equation*}
\CC = \CC_{\Gamma/\Theta, \Lambda/(\Lambda\cap\Theta), \SS},
\end{equation*}
where $\Theta$ is the subgroup consisting in all elements $g$ such that $\beta_{g}^{u}\in \CC$ (note that there is a slight mistake in the reference, since in general $\Theta$ is not contained in $\Lambda$, hence one has to consider $\Lambda/(\Lambda\cap\Theta)$ instead of $\Theta$), which turns out to be normal. By assumption, $\Lambda\subset \Theta$ so that by normality $\overline{\Lambda}\subset \Theta$. In particular, the quotient map $\O(\G_{N}(\CC_{\Gamma, \Lambda, \SS}))\to \O(\G_{N}(\CC))$ factors through $\O(\G_{N}(\CC_{\Gamma/\overline{\Lambda}, \SS}))$.

Moreover, all one-dimensional representations of $\G_{N}(\CC_{\Gamma, \Lambda, \SS})$ are sent to $1$ in $\O(\G_{N}(\CC_{\Gamma/\overline{\Lambda}, \SS}))$ and the linear map corresponding to $\beta_{g}^{u}$ is an intertwiner if and only if the one-dimensional representation corresponding to $g$ is trivial. Thus, the map $\O(\G_{N}(\CC_{\Gamma, \Lambda, \SS}))\to \O(\G_{N}(\CC_{\Gamma/\overline{\Lambda}, \SS}))$ factors through $\O(\G_{N}(\CC))$ and the proof is complete.
\end{proof}

The previous result shows that it is difficult to distinguish in the construction between a subgroup and its normal closure. We will therefore have to assume in the next statement that $\Lambda$ is normal to avoid that issue.

\begin{thm}\label{thm:topological_generation}
Assume that $\Lambda$ is normal. Then, for $N\geqslant 4$, $\widehat{\Gamma}\wr_{\ast, \widehat{\Lambda}}S_{N}^{+}$ is topologically generated by $\widehat{\Gamma}\times S_{N}^{+}$ and $\widehat{\Gamma/\Lambda}\wr_{\ast}S_{N}^{+}$.
\end{thm}

\begin{proof}
It is easy to see that it is enough to prove that for two words $w$ and $w'$ on $\SS$,
\begin{equation*}
\Mor_{\widehat{\Gamma}\wr_{\ast, \widehat{\Lambda}}S_{N}^{+}}(u^{\otimes w}, u^{\otimes w'}) = \Mor_{\widehat{\Gamma}\times S_{N}^{+}}(u^{\otimes w}, u^{\otimes w'})\cap \Mor_{\widehat{\Gamma/\Lambda}\wr_{\ast}S_{N}^{+}}(u^{\otimes w}, u^{\otimes w'}).
\end{equation*}
So let $x$ be an element of the right-hand side. Then, $x$ can be written in two ways as a linear combination of operators $T_{p}$ associated to non-crossing partitions. Because $N\geqslant 4$, these operators are linearly independent (see for instance \cite[Lem 4.16]{freslon2013representation}), hence the partitions appearing in both linear combinations must be the same. In other words, $x$ is a linear combination of operators associated to partitions in $\CC_{\Gamma, \Gamma, \SS}\cap \CC_{\Gamma/\Lambda, \SS}$. We therefore have to prove that this intersection reduces to $\CC_{\Gamma, \Lambda, \SS}$. It is clear that it contains it, since it contains all its generators. Moreover, if follows from this and the proof of \cite[Prop 3.18]{freslon2017noncrossing} that
\begin{equation*}
\CC_{\Gamma, \Gamma, \SS}\cap \CC_{\Gamma/\Lambda, \SS} = \CC_{\Gamma/\Lambda_{0}, \Lambda'/(\Lambda'\cap\Lambda_{0}), \SS}
\end{equation*}
for some normal subgroup $\Lambda_{0}\subset \Gamma$ and a subgroup $\Lambda'$ containing $\Lambda$. Because the intersection is contained in $\CC_{\Gamma, \Gamma, \SS}$, we have that $\Lambda_{0}$ is trivial. Moreover, for any $g\in \Lambda'\setminus\Lambda$, $u^{g}$ is irreducible in $\widehat{\Gamma/\Lambda}\wr_{\ast}S_{N}^{+}$, hence it is also irreducible in $\widehat{\Gamma}\wr_{\ast, \widehat{\Lambda}'}S_{N}^{+}$. But by \cite[Thm 4.4]{freslon2018representation} $u^{g}$ cannot be irreducible if $g\notin \Lambda'$. Thus, $\Lambda' = \Lambda$ and the proof is complete.
\end{proof}

\begin{cor}\label{cor:rfd}
Let $\Gamma$ be a residually finite group and let $\Lambda$ be a normal subgroup such that $\Gamma/\Lambda$ is residually finite. Then, $\widehat{\Gamma}\wr_{\ast, \widehat{\Lambda}}S_{N}^{+}$ is residually finite.
\end{cor}

\begin{proof}
It is clear that a direct product of residually finite compact quantum groups is residually finite, and both $\widehat{\Gamma}$ and $S_{N}^{+}$ are residually finite, the first one by assumption and the second one by \cite[Thm 3.6]{freslon2018topological}. Moreover, it was proven in \cite[Thm 3.11 and Rem 3.21]{freslon2018topological} that the free wreath product of a residually finite discrete group by $S_{N}^{+}$ is residually finite. Eventually, A. Chirvasitu proved in \cite[Cor 2.12]{chirvasitu2015residually} that if a compact quantum group is topologically generated by two residually finite ones, then it is itself residually finite, hence the result.
\end{proof}

The second structure result that we will deduce from the partition description concerns the notion of \emph{monoidal equivalence}. Two compact quantum groups are said to be monoidally equivalent (see \cite{bichon2006ergodic}) if their respective categories of representations are equivalent as monoidal categories. In particular, they have the same representation theory and using the techniques of \cite[Sec 6]{freslon2012examples} one can use monoidal equivalence to transfer results about approximation properties.

In \cite{lemeux2014free}, F. Lemeux and P. Tarrago proved that the free wreath product $\widehat{\Gamma}\wr_{\ast}S_{N}^{+}$ is monoidally equivalent to another compact quantum group which is easier to study explicitely. More precisely, consider the $*$-subalgebra $H$ of $\O(\widehat{\Gamma}\ast SU_{q}(2))$ (we refer the reader for instance to \cite{woronowicz1987twisted} for the definition of quantum $SU(2)$ and to \cite{wang1995free} for the construction of the free product of two compact quantum groups), where $q+q^{-1} = \sqrt{N}$, generated by the coefficients of the representations $u^{1}gu^{1}$ where $u^{1}$ is the fundamental representation of $SU_{q}(2)$ and $g\in \Gamma$ is seen as an irreducible representation of $\widehat{\Gamma}$. It is easily checked that $H$ is a Hopf $*$-subalgebra and if $\HH_{q}(\Gamma)$ denotes the corresponding dual quantum subgroup of $\widehat{\Gamma}\ast SU_{q}(2)$, then $\widehat{\Gamma}\wr_{\ast}S_{N}^{+}$ is monoidally equivalent to $\HH_{q}(\Gamma)$.

We will now prove a refinement of this in our setting, but this requires an extra definition. Consider the $*$-subalgebra $\A\subset \O(SU_{q}(2))$ generated by the coefficients of $u\otimes u$. It is a CQG-algebra and, paralleling the classical case, the corresponding compact quantum group is denoted by $SO_{q}(3)$.

\begin{prop}\label{prop:monoidal_equivalence}
Let $\O(\G_{q})$ be the quotient of $\O(\widehat{\Gamma}\ast SU_{q}(2))$ by the relations making $\O(\widehat{\Lambda})$ commute with $\O(SO_{q}(3))$ and let $\HH_{q}(\Gamma, \Lambda)$ be the image of $\O(\HH_{q}(\Gamma))$ in $\O(\G_{q})$. Then, there is a monoidal equivalence
\begin{equation*}
\widehat{\Gamma}\wr_{\ast, \widehat{\Lambda}}S_{N}^{+} \cong \HH_{q}(\Gamma, \Lambda).
\end{equation*}
\end{prop}

\begin{proof}
First note that the monoidal equivalence constructed in \cite{lemeux2014free} is defined at the level of all partitions and then simply restricted to the category of partitions of $\widehat{\Gamma}\wr_{\ast}S_{N}^{+}$. All we have to do is therefore to look at the image of the extra generators of $\CC_{\Gamma, \Lambda, \SS}$, that is to say the partitions $\beta_{\lambda}$. Let therefore $\lambda = g_{1}\cdots g_{n}\in \Lambda$ and consider the fattening of the partition $\beta_{\lambda}$ as defined in \cite[Prop 5.2]{lemeux2014free}. It has the following form :
\begin{center}
\begin{tikzpicture}[scale=0.5]
\draw (-6,2)--(-6,0.5);
\draw (-5,2)--(-5,1);
\draw (-4,2)--(-4,1.5);
\draw (-3,2)--(-3,1.5);
\draw (-2,2)--(-2,1);

\draw (6,2)--(6,0.5);
\draw (5,2)--(5,1);
\draw (4,2)--(4,1.5);
\draw (3,2)--(3,1.5);
\draw (2,2)--(2,1);

\draw (6,0.5)--(-6,0.5);
\draw (5,1)--(-5,1);
\draw (-4,1.5)--(-3,1.5);
\draw (4,1.5)--(3,1.5);

\draw (0,2)node[below]{$\cdots$};
\draw (-6,2)node[above]{$\bullet$};
\draw (-5,2)node[above]{$g_{1}$};
\draw (-4,2)node[above]{$\bullet$};
\draw (-3,2)node[above]{$\bullet$};
\draw (-2,2)node[above]{$g_{2}$};
\draw (2,2)node[above]{$g_{n-1}$};
\draw (3,2)node[above]{$\bullet$};
\draw (4,2)node[above]{$\bullet$};
\draw (5,2)node[above]{$g_{n}$};
\draw (6,2)node[above]{$\bullet$};

\draw (-6,-2)--(-6,-0.5);
\draw (-5,-2)--(-5,-1);
\draw (-4,-2)--(-4,-1.5);
\draw (-3,-2)--(-3,-1.5);
\draw (-2,-2)--(-2,-1);

\draw (6,-2)--(6,-0.5);
\draw (5,-2)--(5,-1);
\draw (4,-2)--(4,-1.5);
\draw (3,-2)--(3,-1.5);
\draw (2,-2)--(2,-1);

\draw (6,-0.5)--(-6,-0.5);
\draw (5,-1)--(-5,-1);
\draw (-4,-1.5)--(-3,-1.5);
\draw (4,-1.5)--(3,-1.5);

\draw (0,-2)node[above]{$\cdots$};
\draw (-6,-2)node[below]{$\bullet$};
\draw (-5,-2)node[below]{$g_{1}$};
\draw (-4,-2)node[below]{$\bullet$};
\draw (-3,-2)node[below]{$\bullet$};
\draw (-2,-2)node[below]{$g_{2}$};
\draw (2,-2)node[below]{$g_{n-1}$};
\draw (3,-2)node[below]{$\bullet$};
\draw (4,-2)node[below]{$\bullet$};
\draw (5,-2)node[below]{$g_{n}$};
\draw (6,-2)node[below]{$\bullet$};

\draw (-7,0)node[left]{$p_{\lambda} = $};
\end{tikzpicture}
\end{center}
where the $\bullet$ symbols stand for the fundamental representation $u^{1}$ of $SU_{q}(2)$. Let $b_{\lambda}$ be the partition obtained by taking the upper row and removing the outermost pair partition. That partition encodes the representation $g_{1}\otimes \varepsilon_{SU_{q}(2)}\otimes g_{2}\otimes \cdots \otimes g_{n-1}\otimes \varepsilon_{SU_{q}(2)}\otimes g_{n} = \lambda$ of the free product $\widehat{\Gamma}\ast SU_{q}(2)$. Moreover, adding the intertwiner corresponding to $p_{\lambda}$ to the representation category of $\widehat{\Gamma}\ast SU_{q}(2)$ is equivalent to adding the intertwiner corresponding to the rotated version
\begin{center}
\begin{tikzpicture}[scale=0.5]
\draw (-0.5,1)--(-0.5,-1);
\draw (0.5,1)--(0.5,-1);

\draw (-0.5,1)node[above]{$\bullet$};
\draw (-0.5,-1)node[below]{$\bullet$};
\draw (0.5,1)node[above]{$\bullet$};
\draw (0.5,-1)node[below]{$\bullet$};

\draw (-2,0)node{$b_{\lambda}\:\otimes $};
\draw (2, 0)node{$\otimes\: b_{\lambda}^{*}$};
\end{tikzpicture}
\end{center}
Adding this is in turn equivalent to requiring $\lambda$ to commute with all coefficients of $u\otimes u$, where $u$ is the fundamental representation of $SU_{q}(2)$. Because the subalgebra generated by the coefficients of $u\otimes u$ is exactly $\O(SO_{q}(3))$, we conclude that the essential image of the representation category of $\widehat{\Gamma}\wr_{\ast, \widehat{\Lambda}}S_{N}^{+}$ under the fully faithful functor of \cite[Thm 5.11]{lemeux2014free} is exactly the representation category of $\HH_{q}(\Gamma, \Lambda)$.
\end{proof}

\begin{rem}
There is a more general notion of a free wreath product of a compact quantum group by a quantum group acting on a finite-dimensional C*-algebra given by P. Tarrago and J. Wahl in \cite{tarrago2018free}. In particular, the image of $\widehat{\Gamma}\wr_{\ast}S_{N}^{+}$ under the previous monoidal equivalence is in fact $\widehat{\Gamma}\wr_{\ast}SO_{q}(3)$. For that reason, it would be interesting to try to generalize \cite{tarrago2018free} and check whether $\HH_{q}(\Gamma, \Lambda) = (\widehat{\Gamma}, \widehat{\Lambda})\wr_{\ast}SO_{q}(3)$ in some sense.
\end{rem}

As a corollary, we can establish some approximation properties for free wreath products of pairs, provided we are able to prove them on $\HH_{q}(\Gamma, \Lambda)$. Here is an example of such an instance concerning the Haagerup property, the definition of which we recall (see for instance \cite{daws2014haagerup} for details on that notion).

\begin{de}
A compact quantum group $\G$ is said to have the \emph{central Haagerup property} if there exists a net $(\varphi_{i})_{i\in I}$ of functions on $\Irr(\G)$ such that
\begin{itemize}
\item $\varphi_{i}(\alpha)\displaystyle\underset{i}{\longrightarrow} \dim(\alpha)$ for all $\alpha\in \Irr(\G)$ ;
\item For any $\epsilon > 0$, there exists a finite subset $F\subset \Irr(\G)$ such that for any $\alpha\notin F$, $\vert\varphi_{i}(\alpha)\vert < \epsilon\dim(\alpha)$ ;
\item The map $\varphi_{i}$ extend to a state (that is, a positive linear functional such that $\varphi_{i}(1) = 1$) on $\O(\G)$ for all $i\in I$.
\end{itemize}
\end{de}

\begin{cor}\label{cor:haagerup}
Assume that $\Gamma$ has the Haagerup property and that $\Lambda$ is finite. Then, $\widehat{\Gamma}\wr_{\ast, \widehat{\Lambda}}S_{N}^{+}$ has the central Haagerup property.
\end{cor}

\begin{proof}
First, by averaging over $\Lambda$ we may assume that there is a net of functions $\varphi_{i} : \Gamma\to \C$, for $i\in I$, witnessing the central Haagerup property for $\widehat{\Gamma}$ and such that for any $g\in \Gamma$ and $\lambda\in \Lambda$, $\varphi_{i}(\lambda g) = \varphi_{i}(g)$. Let now $(\psi_{j})_{j\in J}$ be the net of functions on $\Irr(SU_{q}(2))$ witnessing the central Haagerup property constructed in \cite[Sec 3]{freslon2013ccap}. We define a map $\eta_{ij}$ on $\Irr(\widehat{\Gamma}\ast SU_{q}(2))$ by the formula
\begin{equation*}
\eta_{ij}\left(\prod_{l=1}^{k} g_{l}u^{n_{l}}\right) = \left(\prod_{l=1}^{k}\varphi_{i}(g_{l})\right)\left(\prod_{l=1}^{k}\psi_{j}(u^{n_{l}})\right).
\end{equation*}
The extension of these maps to $\O(\widehat{\Gamma}\ast SU_{q}(2))$ is nothing but the conditional free product of the extensions of $\varphi_{i}$ and $\psi_{j}$ (see for instance \cite[Thm 7.7]{daws2014haagerup}) and is therefore known to be a state. The net $(\eta_{ij})_{ij\in I\times J}$ therefore witnesses the central Haagerup property for the free product and moreover has the following property : if $\lambda\in \Lambda$, then
\begin{equation*}
\eta_{ij}\left(g_{1}u^{n_{1}}\cdots g_{i-1}u^{n_{i-1}}(\lambda g_{i})u^{n_{i}}\cdots g_{n}u^{n_{k}}\right) = \eta_{ij}\left(g_{1}u^{n_{1}}\cdots (g_{i-1}\lambda)u^{n_{i-1}} g_{i}u^{n_{i}}\cdots g_{n}u^{n_{k}}\right)
\end{equation*}
This follows directly from the equivariance property of $\varphi_{i}$. As a consequence, $\eta_{ij}$ factors through the quotient to $\HH_{q}(\Gamma, \Lambda)$ and witnesses the central Haagerup property there, hence the result by \ref{prop:monoidal_equivalence} and \cite[Sec 6]{freslon2012examples}.
\end{proof}

The same method applies to another approximation property called \emph{weak amenability}. We will not introduce it but simply state the result for the records (see \cite{freslon2012examples} for details).

\begin{cor}
Assume that $\Lambda_{cb}(\Gamma) = 1$ and that $\Lambda$ is finite. Then,
\begin{equation*}
\Lambda_{cb}\left(\widehat{\Gamma}\wr_{\ast, \widehat{\Lambda}}S_{N}^{+}\right) = 1.
\end{equation*}
\end{cor}

\begin{proof}
We can average the mutlipliers witnessing weak amenability of $\Gamma$ to make them $\Lambda$-invariant just as above. Then, the corresponding multipliers implementing weak amenability on $\widehat{\Gamma}\ast SU_{q}(2)$ through the construction of \cite{freslon2014permanence} have the same invariance property as the multipliers $\eta_{ij}$ in the proof of Corollary \ref{cor:haagerup} and the result follows.
\end{proof}

\section{Decomposition results}\label{sec:extension}

As we have seen in Corollary \ref{cor:rfd}, when $\Lambda$ is normal properties of $H_{N}^{+}(\Gamma, \Lambda)$ can have strong connections to properties of $\Lambda$ and of the quotient $\Gamma/\Lambda$. We will now describe that connection from another point of view. Let us start with an alternative expression of the normality of $\Lambda$, namely the fact that it fits into a short exact sequence
\begin{equation*}
1\longrightarrow \Lambda\longrightarrow \Gamma\longrightarrow \Gamma/\Lambda\longrightarrow 1.
\end{equation*}
It is quite natural to wonder whether this induces a corresponding decomposition of $H_{N}^{+}(\Gamma, \Lambda)$, and answering that question first requires recalling some facts about exact sequences of compact quantum groups.

Let $\G$ be a compact quantum group and let $\HH$ be a dual quantum subgroup of $\G$, so that $\O(\HH)\subset \O(\G)$ is a Hopf $*$-subalgebra. If we denote by $I$ the kernel of the counit of $\O(\HH)$, then $I\O(\G)$ and $\O(\G)I$ are both ideals in $\O(\G)$. If they turn out to be equal, then it is a Hopf $*$-ideal and the short sequence
\begin{equation*}
\C\longrightarrow \O(\HH)\longrightarrow \O(\G)\longrightarrow \O(\G)/\O(\G)I\longrightarrow \C
\end{equation*}
is then said to be exact.

\begin{rem}
Our definition is seemingly weaker than the usual one for exact sequences of Hopf algebras in \cite{andruskiewitsch1995extensions}. We are here taking advantage of the fact that CQG-algebras are co-semisimple, hence faithfully flat over their Hopf subalgebras as proven in \cite{chirvasitu2014cosemisimple}, so that the condition above implies that $\O(\HH)$ coincides with both the left and right coinvariants of the surjection (see \cite[Thm 1]{takeuchi1979relative} for a proof).
\end{rem}

Back to our initial situation, it was proven in \cite[Lem 4.3]{freslon2018representation} that the group of one-dimensional representations of $\O(H_{N}^{+}(\Gamma, \Lambda))$ naturally identifies with $\Lambda$. This provides us with a canonical copy of $\C[\Lambda] = \O(\widehat{\Lambda})$ inside $\O(H_{N}^{+}(\Gamma, \Lambda))$. As expected, this is the correct object to consider.

\begin{prop}\label{prop:exactsequence}
Assume that $\Lambda$ is normal in $\Gamma$. Then, there is an exact sequence of Hopf $*$-algebras
\begin{equation*}
\C\longrightarrow \O(\widehat{\Lambda})\longrightarrow \O(H_{N}^{+}(\Gamma, \Lambda))\longrightarrow \O(H_{N}^{+}(\Gamma/\Lambda))\longrightarrow \C.
\end{equation*}
\end{prop}

\begin{proof}
Any element of $\O(H_{N}^{+}(\Gamma, \Lambda))$ being a linear combination of coefficients of irreducible representations, it is enough to consider one such coefficient. The representation theory of $H_{N}^{+}(\Gamma, \Lambda)$ computed in \cite[Thm 4.4]{freslon2018representation} (it is also recalled just before Theorem \ref{thm:gluedproduct}) shows that any irreducible representation of dimension strictly greater than one is of the form
\begin{equation*}
u^{g_{1}}\cdots u^{g_{n}}
\end{equation*}
where $u^{g_{i}} = a(g_{i})$ if $g_{i}\notin\Lambda$ and is the complement of the unique non-trivial one-dimensional subrepresentation of $a(g_{i})$ otherwise. As a consequence, it is enough to check that $u^{g}_{ij}I = Iu^{g}_{ij}$. Note moreover that $I$ is the linear span of the elements $\lambda_{1} - \lambda_{2}$ for $\lambda_{1}, \lambda_{2}\in \Lambda$.

Let us therefore consider an element of the form $x = (\lambda_{1} - \lambda_{2})u^{g}_{ij}\in I\O(H_{N}^{+}(\Gamma, \Lambda))$. By the fusion rules, this equals $u^{\lambda_{1}g}_{ij} - u^{\lambda_{2}g}_{ij}$. Now by normality of $\Lambda$, there exists $\lambda'_{1}, \lambda'_{2}\in \Lambda$ such that $\lambda_{1}g = g\lambda_{1}'$ and $\lambda_{2}g = g\lambda'_{2}$. It follows that
\begin{equation*}
x = u^{g\lambda'_{1}}_{ij} - u^{g\lambda'_{2}}_{ij} = u^{g}_{ij}(\lambda'_{1} - \lambda'_{2})\in \O(H_{N}^{+}(\Gamma, \Lambda))I
\end{equation*}
and the proof is complete.
\end{proof}

As a first consequence of this result, we can bound the cohomological dimension (in the sense of Hochschild cohomology, see for instance \cite[Chap 4]{witherspoon2019hochschild}) of $H_{N}^{+}(\Gamma, \Lambda)$.

\begin{cor}
If $\Lambda$ is normal in $\Gamma$, then
\begin{equation*}
\max\left(3, \cd(\Gamma^{\ast_{\Lambda}^{N}}))\right) \leqslant \cd(H_{N}^{+}(\Gamma, \Lambda)) \leqslant \max\left(3, \cd(\Gamma/\Lambda)\right) + \cd(\Lambda).
\end{equation*}
In particular, if $\Gamma/\Lambda$ is finite, then
\begin{equation*}
\cd(H_{N}^{+}(\Gamma, \Lambda)) = 3.
\end{equation*}
\end{cor}

\begin{proof}
It follows from \cite[Prop 3.2]{bichon2016gerstenhaber}, noticing that because $\O(H_{N}^{+}(\Gamma, \Lambda))$ is co-semisimple, the exact sequence is in fact strict by \cite[Thm 2]{takeuchi1979relative}, that
\begin{equation*}
\cd(H_{N}^{+}(\Gamma, \Lambda)) \leqslant \cd(H_{N}^{+}(\Gamma/\Lambda)) + \cd(\Lambda).
\end{equation*}
On then combines this with \cite[Thm 7.4]{bichon2021monoidal} to get the upper bound.

As for the lower bound, first notice that both $\O(\widehat{\Gamma}^{\ast_{\Lambda}^{N}})$ and $\O(S_{N}^{+})$ are Hopf $*$-subalgebras of $\O(H_{N}^{+}(\Gamma, \Lambda))$. We then apply \cite[Prop 3.2]{bichon2016gerstenhaber}, using the facts that $\cd(S_{N}^{+}) = 3$ by \cite[Thm 6.5]{bichon2016gerstenhaber}.

As for the last statement, first recall that finite groups have cohomological dimension $0$ (note that we are considering the dimension of the complex group ring, not the integer one). Therefore, the right-hand side of the inequalities reduces to $3$, and the result follows.
\end{proof}

\begin{rem}
It is know (see \cite[Cor 5.3]{bichon2018cohomological}) that for a CQG-algebra $A$, $\cd(A^{\ast N}) = \cd(A)$. There does not however not seem to be a similar relation for iterated amalgamated free product. In particular, it is not clear how to relate $\cd(\Gamma^{\ast_{\Lambda}^{N}})$ to $\cd(\Gamma/\Lambda)$ in general.
\end{rem}

There is no hope to get further information on the structure of $H_{N}^{+}(\Gamma, \Lambda)$ without strengthening the assumption on $\Lambda$. But it would be natural to consider now the case where the exact sequence is split, so that $\Gamma$ decomposes as a semi-direct product. We will see that a similar decomposition then holds for the free wreath products in terms of smash product of Hopf algebras. Instead of introducing the general theory, let us simply define the important object in our setting.

Assume that there exists a group homomorphism $\rho : \Gamma/\Lambda\to \Gamma$ such that $\pi\circ\rho = \id$, where $\pi : \Gamma\to \Gamma/\Lambda$ is the quotient map. Consider the vector space $H = \O(\widehat{\Gamma})\otimes \O(H_{N}^{+}(\Gamma/\Lambda))$ and equip it with the unique multiplication such that
\begin{equation*}
(\lambda\otimes u^{g}_{ij})(\mu\otimes u^{h}_{i'j'}) = (\lambda\rho(g)\mu\rho(g)^{-1})\otimes u^{g}_{ij}u^{h}_{i'j'}
\end{equation*}
and with the unique comultiplication such that
\begin{equation*}
\D(\lambda\otimes u^{g}_{ij}) = \sum_{k=1}^{N}\lambda \otimes u^{g}_{ik}\otimes \lambda\otimes u^{g}_{kj}.
\end{equation*}
Then, $H$ is a CQG-algebra and the corresponding compact quantum group will be denoted by $\widehat{\Gamma}\:\sharp\: H_{N}^{+}(\Gamma, \Lambda)$.

\begin{thm}\label{thm:smashproduct}
Let $\Gamma$ be a discrete group and let $\Lambda$ be a normal subgroup such that the corresponding short exact sequence is split. Then, there is an isomorphism
\begin{equation*}
H_{N}^{+}(\Gamma, \Lambda) \cong \widehat{\Gamma}\:\sharp\: H_{N}^{+}(\Gamma/\Lambda).
\end{equation*}
\end{thm}

\begin{proof}
First observe that, in a way similar to the comment in the beginning of the proof of Proposition \ref{prop:directproduct}, if $\Lambda_{1} < \Gamma_{1}$ and $\Lambda_{2} < \Gamma_{2}$ are two pairs of discrete groups and $\rho : \Gamma_{1}\to \Gamma_{2}$ is a group homomorphism such that $\rho(\Lambda_{1})\subset \Lambda_{2}$, then there is a $*$-homomorphism $\O(H_{N}^{+}(\Gamma_{1}, \Lambda_{1}))\to \O(H_{N}^{+}(\Gamma_{2}, \Lambda_{2}))$ sending $a(g)_{ij}$ to $a(\rho(g))_{ij}$.

Let $\rho$ be the splitting of the exact sequence and consider the corresponding map
\begin{equation*}
\gamma : \O(H_{N}^{+}(\Gamma/\Lambda))\to \O(H_{N}^{+}(\Gamma, \Lambda)).
\end{equation*}
Because it is a $*$-homomorphism, it is convolution invertible with inverse $\gamma^{-1} = f\circ S$ (where $S$ denotes the antipode). In other words, the extension of compact quantum groups is \emph{cleft} with cleaving map $\gamma$ and it follows from the general theory (see for instance \cite[Thm 7.2.3]{montgomery1993hopf}) that $H_{N}^{+}(\Gamma, \Lambda)$ is a cocycle smash product. We will now prove that this cocycle smash product coincides with the definition above.

First observe that the corresponding cocycle is trivial because $\gamma$ is an algebra homomorphism. Furthermore, the measuring corresponding to the cleaving map $\gamma$ is given by
\begin{align*}
u^{g}_{ij}\triangleright\lambda & = \sum_{k=1}^{N}\gamma(u^{g}_{ik}\lambda\gamma^{-1}(u^{g}_{kj}) \\
& = \sum_{k=1}^{N}u^{\rho(g)\lambda}_{ik}u^{\rho(g)^{-1}}_{jk}) \\
& = \delta_{ij}\sum_{k=1}^{N}u^{\rho(g)\lambda\rho(g)^{-1}}_{ik} \\
& = \delta_{ij}\rho(g)\lambda\rho(g)^{-1}.
\end{align*}
Now, the algebra structure on $\O(\widehat{\Gamma})\otimes \O(H_{N}^{+}(\Gamma/\Lambda))$ is given by the formula
\begin{equation*}
(\lambda\otimes u^{g}_{ij})(\mu\otimes u^{h}_{i'j'}) = \sum_{k=1}^{N}\lambda\left(u_{ik}^{g}\triangleright \mu)\right)\otimes u^{g}_{kj}u^{h}_{i'j'} = \lambda\rho(g)\mu\rho(g)^{-1}\otimes u^{g}_{ij}u^{h}_{i'j'}.
\end{equation*}
\end{proof}

Even though Theorem \ref{thm:smashproduct} gives a decomposition in terms of $\Lambda$ and $\Gamma/\Lambda$, it cannot be used directly to improve results concerning approximation properties, because even for discrete groups it is know that the Haagerup property for instance does not pass to general semi-direct products. Let us therefore work out a special case where everything works smoothly, namely the case where the smash product reduces to a direct product.

\begin{cor}
There is an isomorphism
\begin{equation*}
H_{N}^{+}(\Gamma, \Lambda)\cong \widehat{\Lambda}\times H_{N}^{+}(\Gamma/\Lambda)
\end{equation*}
if and only if $\Gamma$ is of the form $\Lambda\underset{\Lambda_{0}}{\times} \Gamma_{0}$.
\end{cor}

\begin{proof}
If the isomorphism holds, then by the definition of $\widehat{\Lambda}\:\sharp\: H_{N}^{+}(\Gamma, \Lambda)$, $\Lambda$ must commute with $\Gamma_{0} = \rho(\Gamma/\Lambda)$. Since these two groups generate $\Gamma$, it follows that $\Gamma$ is a quotient of $\Lambda\times \Gamma_{0}$. Let us denote by $\pi$ the quotient map. Because $\pi$ is injective on $\Lambda$, for any $\lambda\in \Lambda$ there exists at most one $\gamma\in \Gamma_{0}$ such that $(\lambda, \gamma)\in \ker(\pi)$, namely $\lambda^{-1}$. As a consequence, if
\begin{equation*}
\Lambda_{0} = \{\lambda\in \Lambda\mid (\lambda, \gamma)\in \ker(\pi)\},
\end{equation*}
then the map $\lambda\mapsto \gamma^{-1}$ is well-defined isomorphism between $\Lambda_{0}$ and a subgroup $\Lambda_{1}\subset \Gamma_{0}$ and $\pi$ is the quotient map identifying these two isomorphic groups. The converse is clear. 
\end{proof}

In particular, under the previous assumptions, if both $\Lambda$ and $\Gamma/\Lambda$ are residually finite (resp. have the Haagerup property, resp. are weakly amenable) then so does $H_{N}^{+}(\Gamma, \Lambda)$.

\bibliographystyle{smfplain}
\bibliography{../../../quantum}

\end{document}